\documentclass[11pt,a4paper]{article}
\pdfoutput=1 % for arxiv

\newcommand{\op}[1]{\textcolor{red}{OP: #1}}

\usepackage{epsf,epsfig,amsfonts,amsgen,amsmath,amstext,amsbsy,amsopn,amsthm,cases,listings,color
}
\usepackage{ebezier,eepic}
\usepackage{color}
\usepackage{multirow}
\usepackage{epstopdf}
\usepackage{graphicx}
\usepackage{pgf,tikz}
\usepackage{mathrsfs}
\usepackage[marginal]{footmisc}
\usepackage{enumitem}
\usepackage[titletoc]{appendix}
\usepackage{booktabs}
\usepackage{url}
\usepackage{relsize}
\usepackage{mathtools}
\usepackage{comment}
\usepackage{pgfplots}
\usepackage{array}
\usepackage{authblk}
\usepackage{amssymb}
\usepackage{float}
\usepackage{wasysym}

\usepackage{empheq}

\usepackage{dsfont}

\usepackage{tikz}
\usepackage{longtable}
\usepackage{subfigure}
\pgfplotsset{compat=1.17}
\usepackage{mathrsfs}
\usepackage{wasysym}
\usetikzlibrary{arrows}
\usepackage{aligned-overset}
\usepackage{bm}%for vector
\usepackage{bbm}%for vector
\usepackage[T1]{fontenc}%for polish hook
\usepackage[backref=page]{hyperref}
\allowdisplaybreaks[1]

\definecolor{uuuuuu}{rgb}{0.27,0.27,0.27}
\definecolor{sqsqsq}{rgb}{0.1255,0.1255,0.1255}

\newtheorem{definition}{Definition} [section]
\newtheorem{theorem}[definition]{Theorem}
\newtheorem{lemma}[definition]{Lemma}

\newtheorem{claim}[definition]{Claim}
\newtheorem{problem}[definition]{Problem}

\newtheorem{fact}[definition]{Fact}

\newenvironment{pf}{\noindent{\bf Proof}}

\newcommand{\vol}{\operatorname{vol}}
\newcommand{\dist}{\operatorname{dist}}

\newcommand{\Mod}[1]{\ \mathrm{mod}\ #1}
\newcommand{\hide}[1]{}

\begin{document}
\title{\bf\Large New upper bound for lattice covering by spheres}
\date{\today}
\author[1]{Jun Gao\thanks{Research supported by ERC Advanced Grant 101020255. Email: \texttt{gj950211@gmail.com}}}
\author[2]{Xizhi Liu\thanks{Research supported by ERC Advanced Grant 101020255 and the Excellent Young Talents Program (Overseas) of the National Natural Science Foundation of China. Email: \texttt{liuxizhi@ustc.edu.cn}}}
\author[1]{Oleg Pikhurko\thanks{Research supported by ERC Advanced Grant 101020255. Email: \texttt{o.pikhurko@warwick.ac.uk}}}
\author[1]{Shumin Sun\thanks{Research supported by ERC Advanced Grant 101020255. Email: \texttt{Shumin.Sun@warwick.ac.uk}}%
}
\affil[1]{Mathematics Institute and DIMAP,
            University of Warwick,
            Coventry, UK}
\affil[2]{School of Mathematical Sciences, 
            USTC,
            Hefei, 230026, China}
        % USTC = University of Science and Technology of China
%
\maketitle
%%%%%%%%%%%%%%%%%%%%%%%%%%%%%%%%%%%%%%%%%%%%%
\begin{abstract}
We show that there exists a lattice covering of $\mathbb{R}^n$ by Eucledian spheres of equal radius  with density $O\big(n \ln^{\beta} n \big)$ as $n\to\infty$, where 
\begin{align*}
    \beta \coloneqq \frac{1}{2} \log_2 \left(\frac{8 \pi \mathrm{e}}{3\sqrt 3}\right)=1.85837...\,.
\end{align*}
This improves upon the previously best known upper bound by Rogers from 1959 of $O\big(n \ln^{\alpha} n \big)$, where $\alpha \coloneqq \frac{1}{2} \log_{2}(2\pi \mathrm{e})=2.0471...$\,.  
    
% \medskip

% \noindent\textbf{Keywords:} lattice covering, sphere covering

% \medskip

% \noindent\textbf{MSC2010:} 05C35, 05D05, 05C65
\end{abstract}
%%%%%%%%%%%%%%%%%%%%%%%%%%%%%%%%%%%%%
\section{Introduction}
Given $n$, we would like to cover the  entire space $\mathbb{R}^{n}$ by placing spheres\footnote{Throughout this work, we adopt the convention that `sphere' means a closed Euclidean ball.} of the same radius $r$ at each element of a lattice $\Lambda$, that is, we require that 
\begin{align}\label{equ:lattice-covering-def}
   \Lambda + B_{r}^{n}= \mathbb{R}^{n},
 %     \coloneqq \left\{ \mathbf{x} + \mathbf{y} = \Lambda + B_{r}^{n},
    %\colon \text{$\mathbf{x} \in \Lambda$ and $\mathbf{y} \in B_{r}^{n}$} \right\}.
\end{align}
 where  $B_{r}^{n}:=\{\mathbf{x}\in \mathbb{R}^{n}\colon \|\mathbf{x}\|_2\le r\}$ denotes the Euclidean sphere of radius $r$  in $\mathbb{R}^n$ centred at the origin and $X+Y\coloneqq \left\{ \mathbf{x} + \mathbf{y}\colon \text{$\mathbf{x} \in X$ and $\mathbf{y} \in Y$} \right\}$ denotes the \emph{sum} of two sets $X,Y\subseteq \mathbb{R}^{n}$.
 We call any such pair $(\Lambda, B_{r}^{n})$  a \emph{(sphere) lattice covering} of $\mathbb{R}^{n}$ and define its \emph{density}  as 
\begin{align*}
    \Theta(\Lambda, B_{r}^{n}) 
    \coloneqq \frac{\vol(B^n_r)}{|\det(\Lambda)|},
\end{align*}
where $\vol(B^n_r)$ denotes the volume of $B^n_r$ and $\det(\Lambda)$ is the \emph{determinant} of $\Lambda$ which can be defined as
\begin{align*}
    \det(\Lambda)
    \coloneqq \det\left[\mathbf{b}_{1}, \ldots, \mathbf{b}_{n}\right], 
\end{align*}
the determinant of the matrix made of some (equivalently, any) linearly independent vectors\footnote{Unless otherwise specified, all vectors in this work are considered column vectors.} $\mathbf{b}_{1}, \ldots, \mathbf{b}_{n}  \in \mathbb{R}^{n}$ that \emph{generate} the lattice $\Lambda$, i.e., satisfy
\begin{align}\label{eq:generate}
    \Lambda
    = \left\{\lambda_1 \mathbf{b}_{1} + \cdots + \lambda_n \mathbf{b}_{n} \colon \text{$\lambda_i \in \mathbb{Z}$ for $i \in [n]$}\right\},  
\end{align}
where $[n]:=\{1,\ldots,n\}$.
\hide{\op{I feel that it is rather unusual for a paper using linear algebra, to use row vectors. Can we change to column vectors, please?}}
The \emph{covering density} of $\Lambda$ is then defined as 
\begin{align*}
    \Theta(\Lambda)
    \coloneqq \min_{r \ge 0}\left\{\Theta(\Lambda, B_{r}^{n})  \colon \mathbb{R}^{n} 
    = \Lambda + B_{r}^{n} \right\}. 
\end{align*}
The classical \emph{lattice covering problem}, a central topic in combinatorial geometry (see e.g. books~\cite{Rog64,GL87,CS99book}), asks for the \emph{optimal lattice covering density} in dimension $n$, defined as 
\begin{align*}
    \Theta_n
    \coloneqq \inf\left\{ \Theta(\Lambda) \colon \text{$\Lambda \subseteq \mathbb{R}^n$ is a lattice} \right\}.
\end{align*}
Determining $\Theta_n$ seems  a very difficult problem, with exact values known only for $n\le 5$ (see~\cite{Ker39,Bam54,Bar56,Few56,DR63,RB78}) and with many questions (such as, for example, whether the Leech lattice is optimal) being still open. Various lower and upper bounds for $\Theta_n$ were obtained in a large number of works, starting with the classical papers~\cite{BD52,Dav52,ER53,59R,CFR59} from the 1950s; we refer the reader to the papers~\cite{SV06,EdelsbrunnerKerber11} that contain  overviews of more recent results. 
%Also, the version where the sphere is replaced by an arbitrary convex body has been actively studied, see e.g.\ the recent paper~\cite{schymura2025lattice} for further references.

\hide{OP: I removed it since our paper does not focus on small dimensions
\setlength{\tabcolsep}{15pt}
\setlength{\arrayrulewidth}{0.4mm}
\begin{table}[h!]
\centering
\renewcommand{\arraystretch}{1.5} 
\begin{tabular}{c|c|>{\centering\arraybackslash}p{5.5cm}}
\hline
$n$ & $\Theta_n$ & References \\
\hline
2 & ${2\pi}/{3\sqrt{3}} = 1.2092\ldots $ & Kershner~\cite{Ker39} \\
\hline
3 & ${5\sqrt{5}\pi}/{24} = 1.4635\ldots $ & Bambah~\cite{Bam54}, Barnes~\cite{Bar56}, Few~\cite{Few56} \\
\hline
4 & ${2\pi^2}/{5\sqrt{5}} = 1.76553\ldots $ & Delone--Ry{\v s}kov~\cite{DR63} \\
\hline
5 & ${245\sqrt{35}\pi^2}/{3888\sqrt{3}} = 2.12429\ldots $ & Ry{\v s}kov--Baranovski{\u i}~\cite{RB78} \\
\hline
\end{tabular}
\caption{Known exact values of $\Theta_n$.}
\end{table}
}

In general, for any convex body $K \subseteq \mathbb{R}^n$, one can similarly define the optimal lattice covering density $\Theta_{n,K}$ of $K$ (see e.g.~\cite{Rog64} for details). Improving upon Rogers’~\cite{59R} upper bound $\Theta_{n,K} = O\left(n^{\log_2 \ln n + O(1)}\right)$ from 1959, a recent breakthrough by Ordentlich--Regev--Weiss~\cite{ORW22} shows that $\Theta_{n,K} = O(n^2)$ holds universally for all convex bodies $K \subseteq \mathbb{R}^n$.

However, in perhaps the most fundamental case when $K$ is the sphere, this general upper bound does not improve upon Rogers’ other result from the same paper~\cite{59R} that $\Theta_n = O\left(n \ln^{\alpha} n\right)$, where $\alpha \coloneqq \tfrac{1}{2} \log_2(2\pi \mathrm{e}) = 2.0471 ...$\,.

In this work, we establish the following upper bound for $\Theta_n$, improving upon the above mentioned bound of Rogers~\cite{59R}.
\begin{theorem}\label{THM:main-improve-Rogers}
    There exists a constant $C$ such that for every integer $n\ge 1$, it holds that
    \begin{align*}
        \Theta_n
        \le C n \ln^{\beta} n,
        \quad\text{where}\quad 
        \beta \coloneqq \frac{1}{2} \log_2 \left(\frac{8 \pi \mathrm{e}}{3\sqrt 3}\right)
        = 1.85837...\,. 
    \end{align*}
\end{theorem}
Let us remark that the factor $n$ in Theorem~\ref{THM:main-improve-Rogers} is necessary, as shown by Coxeter--Few--Rogers~\cite{CFR59} who proved that $\Theta_n \ge \left(\mathrm{e}^{-3/2} + o(1)\right)n$, improving upon earlier results of Bambah--Davenport~\cite{BD52} and  Erd\H{o}s--Rogers~\cite{ER53}.

Another obstacle to improving the upper bound of $\Theta_n$ is that, even when the condition that $\Lambda \subseteq \mathbb{R}^n$ is a lattice is removed and arbitrary sphere coverings of $\mathbb{R}^n$ are allowed, the best known asymptotic upper bound still has order $n \ln n$ (see e.g.~\cite{57R,62ER,Rog63,BW03,07D}).

Theorem~\ref{THM:main-improve-Rogers} follows relatively quickly from the following more general theorem, which provides a general strategy for proving upper bounds on $\Theta_n$. To state the result, we first introduce some necessary definitions.

Given a point $\mathbf{x} \in \mathbb{R}^n$ and $n$ linearly independent vectors $\mathbf{b}_{1}, \ldots, \mathbf{b}_{n} \in \mathbb{R}^{n}$, the \emph{parallelepiped} $\mathrm{P}=\mathrm{P}_{\mathbf{x}}$ starting at $\mathbf{x} \in \mathbb{R}^{n}$ and generated by $\{\mathbf{b}_{1}, \ldots, \mathbf{b}_{n}\}$ is defined as the convex hull of
\begin{align*}
    V_{\mathbf{x}}(\mathbf{b}_{1}, \ldots, \mathbf{b}_{n})
    \coloneqq \left\{ \mathbf{x} + \lambda_{1} \mathbf{b}_{1} + \cdots + \lambda_{n} \mathbf{b}_{n} \colon \text{$\lambda_i \in \{0,1\}$ for $i \in [n]$} \right\}.
\end{align*}
 Trivially, $V_{\mathbf{x}}(\mathbf{b}_{1}, \ldots, \mathbf{b}_{n})$ is exactly the set of the vertices of the polytope $\mathrm{P}$ and we will refer to this set as $V(\mathrm{P})$.
%The \emph{volume} of $\mathrm{P}$, denoted by $\vol(\mathrm{P})$, is the volume of the convex hull spanned by vertices in $\mathrm{P}$. 
We say a parallelepiped $\mathrm{P} \subseteq \mathbb{R}^{n}$ is a \emph{$\Lambda$-parallelepiped} if $V(\mathrm{P}) \subseteq \Lambda$. 
If, in addition, $\vol(\mathrm{P}) = |\mathrm{det}(\Lambda)|$, then $\mathrm{P}$ is called a \emph{fundamental parallelepiped} of~$\Lambda$. For example, any set of vectors that generates $\Lambda$ as in~\eqref{eq:generate} produces a fundamental parallelepiped.

The following concept will be crucial for our result. 

\begin{definition}[Robust lattice covering]\label{DEF:robust-covering}
    Let $d \ge 1$ be an integer and $r \ge 0$ be a real number. 
    A lattice covering $(\Lambda, B_{r}^{d})$ of $\mathbb{R}^{d}$ is \emph{robust} if every closed ball of radius $r$ in $\mathbb{R}^{d}$ contains a fundamental parallelepiped of $\Lambda$.
\end{definition}

Extending the definition of $\Theta_n$, we define the \emph{optimal robust lattice covering density} of $\mathbb{R}^{n}$ as 
\begin{align*}
    \tilde{\Theta}_n
    \coloneqq \inf \{ \Theta(\Lambda, B_{r}^{n}) \colon \text{$(\Lambda, B_{r}^{n})$ is a robust lattice covering of $\mathbb{R}^n$} \}.
\end{align*}

For every integer $d\ge 1$, define 
\begin{align*}
    \nu_{d}
    \coloneqq \vol\left(B_{\sqrt{d}}^{d}\right)
    = \frac{(\pi d)^{\frac{d}{2}}}{\Gamma\left( \frac{d}{2}+1 \right)}, 
\end{align*}
where $\Gamma$ denotes the gamma function.
%
% Our proof of Theorem~\ref{THM:main-improve-Rogers} relies on the following theorem as a key ingredient.
%The \emph{optimal $d$-robust lattice covering density}, denoted  $F(d)$, is defined as
%\[
%F(d) \coloneqq  \inf \left\{\frac{\vol(B^d_r)}{\det(\Lambda)}  : (\Lambda,B_r^d) \text{ is a $d$-robust lattice covering in } \mathbb{R}^n \right\}.
%\]

The following result provides an asymptotic upper bound for $\Theta_n$ in terms of $\tilde{\Theta}_d$. 
\begin{theorem}\label{THM:general-parallelepiped}
    For every integer $d \ge 1$, there exists a constant $C_{\ref{THM:general-parallelepiped}} = C_{\ref{THM:general-parallelepiped}}(d)$ such that for $n \ge d$, 
    \begin{align*}
        \Theta_n 
        \le C_{\ref{THM:general-parallelepiped}} n \ln^{\gamma} n, 
        \quad\text{where}\quad 
        \gamma = \gamma_{d} \coloneqq \frac{1}{2}\log_{2}(2\pi \mathrm{e}) - \frac{1}{d} \log_{2}\left({\nu_{d}}/{\tilde{\Theta}_d}\right). 
    \end{align*}
\end{theorem}

It is straightforward to verify that $\left( \mathbb{Z}^d, B^{d}_{\sqrt{d}} \right)$ is a robust lattice covering of $\mathbb{R}^{d}$ with density~$\nu_{d}$ for any $d\ge 1$.
Hence, $\Tilde{\Theta}_{d} \le \nu_{d}$.
%, and therefore the ratio $\nu_{d}/{\tilde{\Theta}_d}$ is always at least $1$. 
Thus that $\gamma_{d} \le \frac{1}{2}\log_{2}(2\pi \mathrm{e})$, which recovers Rogers' bound by Theorem~\ref{THM:general-parallelepiped} (applied with any chosen $d\ge 1$).

Theorem~\ref{THM:main-improve-Rogers} follows immediately from Theorem~\ref{THM:general-parallelepiped} and the following upper bound for $\Tilde{\Theta}_{2}$. 
\begin{lemma}\label{lm:robust-covering-R2}
    There exists a robust lattice covering of $\mathbb{R}^{2}$ with density $\frac{8 \pi}{3\sqrt{3}}$. In particular, 
    \begin{align*}
        \tilde{\Theta}_2 
        \le \frac{8 \pi}{3\sqrt{3}}. 
    \end{align*}
\end{lemma}

\begin{proof}[Proof of Lemma~\ref{lm:robust-covering-R2}]
    Let 
    \begin{align*}
        \mathbf{v}_1 \coloneqq (1,0)^{t}, \quad 
        \mathbf{v_2} \coloneqq \big( 1/2, \sqrt{3}/2 \big)^{t}, 
        \quad\text{and}\quad 
        r \coloneqq 2/\sqrt{3}, 
    \end{align*}
    where $\mathbf{v}^t$ denotes the transposition of a vector~$\mathbf{v}$.
    Let $\Lambda \subseteq \mathbb{R}^{2}$ denote the lattice generated by $\{\mathbf{v}_1, \mathbf{v}_{2}\}$ (see Figure~\ref{fig:robust-covering-R2}). 
    We claim that $\left(\Lambda, B^2_r\right)$ is a robust lattice covering of $\mathbb{R}^{2}$.
    By definition, it amounts to showing that for every point $\mathbf{w} \in \mathbb{R}^2$, the sphere $B_{r}^{2}(\mathbf{w})$ of radius $r$ centred at $\mathbf{w}$ contains a fundamental parallelepiped of $\Lambda$.
    By symmetry, it suffices to prove this statement for all points $\mathbf{w}$ contained in the equilateral triangle $\triangle_{ABE}$ shown in Figure~\ref{fig:robust-covering-R2}.

%%%%%%%%%%%%%%%%%%%%%%%%%
\newcommand\parallelogram{
	\mathord{\text{
			\tikz[baseline]
			\draw (0,.1ex) -- (.8em,.1ex) -- (1em,1.6ex) -- (.2em,1.6ex) -- cycle;}}}
%%%%
\begin{figure}[ht]
\centering
\begin{tikzpicture}[scale=1]

  % Invisible bounding box (keeps final PDF symmetric)
  \path[opacity=0] (-2,-2) rectangle (2,2);

  % Basis vectors
  \def\ax{1}          % v1_x
  \def\ay{0}          % v1_y
  \def\bx{0.5}        % v2_x
  \def\by{0.866}      % v2_y = sqrt(3)/2
  \def\radius{1.1547} % 2/sqrt(3)

  % 3×3 lattice points centred at the origin
\foreach \i in {-1,0,1,2,3,4}{
\foreach \j in {-1,0}{
\pgfmathsetmacro\x{\i*\ax + \j*\bx}
\pgfmathsetmacro\y{\i*\ay + \j*\by}
\fill (\x,\y) circle (1pt);
}
}
   \foreach \i in {-2,-1,0,1,2,3}{
    \foreach \j in {1,2}{
      \pgfmathsetmacro\x{\i*\ax + \j*\bx}
      \pgfmathsetmacro\y{\i*\ay + \j*\by}
      \fill (\x,\y) circle (1pt);
    }
  }
     \foreach \i in {-3,-2,-1,0,1,2}{
    \foreach \j in {3}{
      \pgfmathsetmacro\x{\i*\ax + \j*\bx}
      \pgfmathsetmacro\y{\i*\ay + \j*\by}
      \fill (\x,\y) circle (1pt);
    }
  }

  % Ball of radius 2/sqrt(3)
  \draw[cyan,thick] (0,0) circle (\radius);
  \draw[cyan,thick] (2,0) circle (\radius);
  \draw[cyan,thick] (1, 2*\by) circle (\radius);

  % Basis arrows
% \draw[black,thick] (0,0) -- (\ax,\ay)  --  (\bx,\by) -- (\bx -\ax,\by-\ay) -- (0,0);
% \draw[black,thick] (\ax,\ay) --(2*\ax,2*\ay) -- (\bx +2*\ax,\by+2*\ay)-- (\bx +\ax,\by+\ay) -- (\ax,\ay);
% \draw[black,thick] (\bx,\by) --(2*\bx,2*\by) -- (\bx +\ax,\by+\ay)-- (\ax,\ay) -- (\bx,\by);

\draw[gray, thick] (0,0) -- (\bx,\by); 
\draw[gray, thick] (0,0) -- (\ax,\ay); 
\draw[gray, thick] (\bx,\by) -- (\ax,\ay); 
\draw[gray, thick] (\ax+\bx, \ay+\by) -- (\bx,\by); 
\draw[gray, thick] (\ax+\bx, \ay+\by) -- (2*\ax, 2*\ay); 
\draw[gray, thick] (\ax,\ay) -- (2*\ax, 2*\ay); 
\draw[gray, thick] (\ax,\ay) -- (\ax+\bx, \ay+\by); 
\draw[gray, thick] (\bx,\by) -- (2*\bx, 2*\by); 
\draw[gray, thick] (\ax+\bx, \ay+\by) -- (2*\bx, 2*\by); 

\fill (0,0) circle (0.06cm);
\fill (\ax,\ay) circle (0.06cm);
\fill (\bx,\by) circle (0.06cm);
% \fill (-\bx,\by) circle (0.07cm);
\fill (2*\ax,2*\ay) circle (0.06cm);
\fill (\ax+\bx,\ay+\by) circle (0.06cm);
% \fill (2*\ax+\bx,2*\ay+\by) circle (0.07cm);
\fill (2*\bx,2*\by) circle (0.06cm);
\fill[magenta] (\ax,\radius/2) circle (0.06cm);

\node[below] at (0, 0) {\footnotesize $O$};
\node[below] at (\ax-0.03, \ay+0.03) {\footnotesize $A$};
\node[above] at (\bx-0.15, \by-0.1) {\footnotesize $B$};
\node[below] at (2*\ax, 2*\ay) {\footnotesize $D$};
% \node[below] at (-\bx-0.1, \by) {\footnotesize $C$};
\node[above] at (\ax+\bx+0.22, \ay+\by-0.1) {\footnotesize $E$};
\node[above] at (2*\bx, 2*\by) {\footnotesize $G$};
% \node[below] at (2*\ax+\bx+0.1, 2*\ay+\by) {\footnotesize $F$};
\node[below] at (\ax, \radius/2) {\footnotesize $X$};
\end{tikzpicture}
\caption{\hide{\op{The points $C,F$ and bold lines at them are irrelevant. Please re-draw the figure.}} The lattice generated by 
         $\mathbf v_1=(1,0)^{t}$ and 
         $\mathbf v_2=\bigl(1/2,\sqrt3/{2}\bigr)^{t}$, three balls of radius $2/\sqrt3$ centred at the origin $O$, $D=2\mathbf v_1$ and $G= 2\mathbf v_2$, and three different fundamental parallelepipeds $OAEB, DABE, GBAE$. The point $X$ is the centre of the triangle $ABE$.}
\label{fig:robust-covering-R2}
\end{figure}
%%%%%%%%%%%%%%%%%%%%%%%%%

    Let $X$ denote the centre of $\triangle_{ABE}$.  
    It is easy to see that
    \begin{itemize}
        \item if $\mathbf{w} \in \triangle_{AXB}$, then the ball $B_{r}^{2}(\mathbf{w})$ contains the fundamental parallelepiped $\parallelogram_{OAEB}$; 
        \item if $\mathbf{w} \in \triangle_{AXE}$, then the ball $B_{r}^{2}(\mathbf{w})$ contains the fundamental parallelepiped $\parallelogram_{DABE}$; 
        \item if $\mathbf{w} \in \triangle_{BXE}$, then the ball $B_{r}^{2}(\mathbf{w})$ contains the fundamental parallelepiped $\parallelogram_{GBAE}$.
    \end{itemize}
    Therefore, $\left(\Lambda, B^2_r\right)$ is a robust lattice covering of $\mathbb{R}^{2}$.
    The covering density of $(\Lambda, B^2_r)$ is 
    \begin{align*}
        \frac{\vol(B^2_r)}{\det(\Lambda)} 
        = \frac{\left(2/\sqrt{3}\right)^2 \pi}{\sqrt{3}/2} 
        = \frac{8\pi}{3\sqrt{3}}, 
    \end{align*}
    which completes the proof of Lemma~\ref{lm:robust-covering-R2}. 
\end{proof}

% We will present some necessary definitions and preliminary results in the next section.
% In Section~\ref{SEC:proof-mian}, we provide the proof of Theorem~\ref{THM:general-parallelepiped}. Section~\ref{SEC:remarks} concludes with several open problems and remarks.
In the next section, we present the proof of Theorem~\ref{THM:general-parallelepiped}, assuming a key lemma (Lemma~\ref{LEMMA:cartesian-product-covering}) whose proof is deferred to Section~\ref{SEC:proof-lemma-cartesian-product-covering}.
We include some concluding remarks
in Section~\ref{SEC:remarks}.

%%%%%%%%%%%%%%%%%%%%%%%%%%%
\section{Proof of Theorem~\ref{THM:general-parallelepiped}}\label{SEC:proof-mian}
In this section, we present the proof of Theorem~\ref{THM:general-parallelepiped}.
We begin by listing some auxiliary results from Rogers' earlier work~\cite{59R}.

Given a lattice $\Lambda \subseteq \mathbb{R}^{n}$ and a convex body $K \subseteq \mathbb{R}^{n}$, let $\bar{\rho}(\Lambda + K)$ denote the density of the points in $\mathbb{R}^{n}$ that are not covered by $\Lambda + K$. 
\begin{lemma}[{\cite[Lemma~2]{59R}}]\label{LEMMA:almost-cover-Rogers}
    There exist constants $N_{\ref{LEMMA:almost-cover-Rogers}}$ and $C_{\ref{LEMMA:almost-cover-Rogers}}$ such that the following holds for every $n \ge N_{\ref{LEMMA:almost-cover-Rogers}}$. 
    For every convex body $K \subseteq \mathbb{R}^n$, there exists a lattice $\Lambda \subseteq \mathbb{R}^{n}$ with $\det(\Lambda) = {\vol(K)}/{\eta_n}$, where $\eta_n 
        \coloneqq \frac{n}{4} \ln\left( \frac{27}{16} \right) - 3 \ln n$, such that 
    %the set of points in $\mathbb{R}^n$ not covered by $\Lambda +K$ has density at most $C_{\ref{LEMMA:almost-cover-Rogers}} n^3 \left( 16/27 \right)^{n/4}$, where 
    \begin{align}\label{equ:eta-n-def}
        \bar{\rho}(\Lambda + K)
        \le C_{\ref{LEMMA:almost-cover-Rogers}} n^3 \left( \frac{16}{27} \right)^{n/4}.
%        \quad\text{where}\quad        
%        \eta_n 
%        \coloneqq \frac{n}{4} \ln\left( \frac{27}{16} \right) - 3 \ln n.
    \end{align}
\end{lemma}

\begin{lemma}[{\cite[Lemma~4]{59R}}]\label{LEMMA:expand-Rogers}
    Let $K \subseteq \mathbb{R}^n$ be a convex body and $\Lambda \subseteq \mathbb{R}^{n}$ be a lattice. Suppose that $\bar{\rho}(\Lambda + K) \le (n^n+1)^{-1}$. 
    %$\Lambda + K$ covers all of $\mathbb{R}^n$ except for a set of density less than $(n^n+1)^{-1}$. 
    Then $(\Lambda, (1+1/n)K)$ is a lattice covering of $\mathbb{R}^{n}$, that is, $\Lambda + (1+1/n)K=\mathbb{R}^{n}$. 
\end{lemma}

The following lemma, which extends~{\cite[Lemma~3]{59R}}, will be crucial for our proof. 
Due to its technical complexity, we postpone its proof to Section~\ref{SEC:proof-lemma-cartesian-product-covering}.
\begin{lemma}\label{LEMMA:cartesian-product-covering}
Let $d\ge 1$ be an integer for which there exists a robust lattice covering of $\mathbb{R}^{d}$ with density $D\le \nu_d$. Then there is a constant $C_{\ref{LEMMA:cartesian-product-covering}} = C_{\ref{LEMMA:cartesian-product-covering}}(d)$ such that, for any $n\ge 1$, if $K \subseteq \mathbb{R}^{n}$ is a convex body and $\Lambda \subseteq \mathbb{R}^{n}$ is a lattice, then there is a lattice $\tilde{\Lambda}\subseteq \mathbb{R}^{n+d}$ satisfying
  \begin{align*}
        \bar{\rho}\big( \tilde{\Lambda}+\tilde{K}\big) 
        \le C_{\ref{LEMMA:cartesian-product-covering}} \left(\bar{\rho}\big( \Lambda+K\big)\right)^{2^d},
    \end{align*}
where $\tilde{K} \subseteq \mathbb{R}^{n+d}$ denotes the Cartesian product of $K$ and the $d$-dimensional sphere of volume~$D$. 
\end{lemma}

% For a point $\mathbf{x}\in \mathbb{R}^n$, let $B_r^n(\B x)$ denote the $n$-dimensional sphere of radius $r$ centred at $\mathbf{x}$. 
% Recall that we write $B_r^n$ to denote $B_r^n(\B 0)$ for convenience. 

We will also use the following simple fact. 
\begin{fact}\label{FACT:sub-convex-cartesian-product}
    Suppose that $n, d, k \ge 1$ are integers satisfying $1 \le k d \le n$. 
    Then 
    \begin{align*}
        K_{k,d} 
        \coloneqq B_{\sqrt{n-kd}}^{n-kd} \times B_{\sqrt{d}}^{d} \times \cdots \times B_{\sqrt{d}}^{d} 
        %\subseteq B_{\sqrt{n}}^{n}, 
    \end{align*}
    is a subset of $B_{\sqrt{n}}^{n}$
    and 
    \begin{align*}
        \vol(K_{k,d}) 
        = \frac{\nu_{n-kd} \cdot \nu_{d}^{k} }{\nu_{n}} \cdot \vol\left(B^n_{\sqrt{n}}\right). 
    \end{align*}
\end{fact}

Now we present the proof of Theorem~\ref{THM:general-parallelepiped}.
\begin{proof}[Proof of Theorem~\ref{THM:general-parallelepiped}]
    Let $d \ge 1$ be an integer such  that there exists a robust lattice covering of $\mathbb{R}^{d}$ with density $D\le \nu_{d}$. 
    Let $C_{\ref{LEMMA:cartesian-product-covering}} = C_{\ref{LEMMA:cartesian-product-covering}}(d)$ be the constant given by Lemma~\ref{LEMMA:cartesian-product-covering}. 
    Let $C \coloneqq 2\mathrm{e} (2\pi \mathrm{e})^{5d/2}/5$.
    Let $n$ be a sufficiently large integer. 
    Fix an integer $k$ satisfying
    \begin{align}\label{equ:k-choice}
        \frac{1}{d} \log_2\ln n + 4 
        \le k 
        \le \frac{1}{d} \log_2\ln n + 5. 
    \end{align}
    Let
    \begin{align}\label{equ:eta-choice}
        \eta 
        \coloneqq \frac{n-kd}{4} \ln\left( \frac{27}{16} \right) - 3 \ln (n-kd)
        < \frac{n}{5}.
    \end{align}

    We aim to show that there exists a lattice covering $(\Lambda,B^n)$ of $\mathbb{R}^{n}$ with density at most
    \begin{align*}
        C n \left({D}/{\nu_{d}}\right)^{\frac{1}{d}\log_2\ln n} (2\pi \mathrm{e})^{\frac{1}{2}\log_2\ln n}.
    \end{align*}

    Let $K_0 \subseteq \mathbb{R}^{n-kd}$ be a sphere with volume $\eta$ at the origin. 
    Let $r\in \mathbb{R}$ be such that $B_{r}^{d} = D$. 
    For $i \in [k]$, define $K_{i} \coloneqq K_{i-1} \times B_{r}^{d}$.  
    % be the Cartesian product of $K_{i-1}$ and the $d$-dimensional sphere with volume $D$. 
    Note that for $i \in [0, k]$, 
    \begin{align*}
        \vol(K_{i}) = \eta D^i. 
    \end{align*}
    %
    %\op{I think that we do not need translation; if we want to be formal we can say that $K_0$ is the sphere with volume $\eta$ centred at the origin.}
    By Fact~\ref{FACT:sub-convex-cartesian-product}, there exists an $n$-dimensional ball $B \subseteq \mathbb{R}^{n}$ such that, after some linear transformation $T$ (scaling the radii of the balls $K_0$ and $B_{r}^{d}$), the set $K_{k}$ is contained in $B$, and 
    \begin{align}\label{equ:volume-B-upper-bound}
      \vol(B) 
      = \vol(T(K_{k})) \cdot \frac{\nu_{n}}{\nu_{n-kd} \cdot \nu_{d}^{k} }
      = |\mathrm{det}(T)| \cdot \eta  \left(\frac{D}{ \nu_{d} }\right)^k \frac{\nu_{n}}{\nu_{n-kd}}. 
      % & = \eta \cdot D^k \cdot  \frac{\nu_{n}}{\nu_{n-kd} \cdot \nu_{d}^{k} }  \notag \\
      % &  = \eta  \left(\frac{D}{ \nu_{d} }\right)^k \cdot \pi^{\frac{kd}{2}} \cdot \frac{\Gamma\left( \frac{n-dk}{2}+1 \right)}{\Gamma\left( \frac{n}{2}+1 \right)} \cdot\frac{n^{n/2}}{(n-dk)^{(n-dk)/2}} \notag \\
      % & \le 2 \eta  \left(\frac{D}{ \nu_{d} }\right)^k (2\pi \mathrm{e})^{\frac{kd}{2}} 
    \end{align}

    Using the estimate $\Gamma(1+x) = (1+o(1)) \sqrt{2\pi x}\left(x/\mathrm{e}\right)^{x}$ as $x\to\infty$ (see e.g.~\cite{Dav59}), we obtain  
    \begin{align*}
        \nu_{n}
        = \frac{(\pi n)^{\frac{n}{2}}}{\Gamma\left(\frac{n}{2}+1\right)}
        = (1+o(1)) \frac{(2\pi \mathrm{e})^{\frac{n}{2}}}{\sqrt{\pi n}}. 
    \end{align*}
    It follows from~\eqref{equ:volume-B-upper-bound}, together with the assumption that $n$ is sufficiently large, that 
    \begin{align}\label{equ:volume-B-upper-bound-b}
      \vol(B) 
      & = (1+o(1)) |\mathrm{det}(T)| \cdot \eta  \left(\frac{D}{ \nu_{d} }\right)^k \frac{(2\pi \mathrm{e})^{\frac{n}{2}} \sqrt{\pi (n-kd)}}{(2\pi \mathrm{e})^{\frac{n-kd}{2}} \sqrt{\pi n}}  \notag \\
      & \le 2 |\mathrm{det}(T)| \cdot \eta  \left(\frac{D}{ \nu_{d} }\right)^k (2\pi \mathrm{e})^{\frac{kd}{2}}. 
    \end{align}
    Let $C_{\ref{LEMMA:almost-cover-Rogers}}$ be the constant given by Lemma~\ref{LEMMA:almost-cover-Rogers}. 
    Applying Lemma~\ref{LEMMA:almost-cover-Rogers} to $K_0$, we obtain a lattice $\Lambda_0 \subseteq \mathbb{R}^{n-kd}$ with $\det(\Lambda_{0}) = \vol(K_{0})/\eta = 1$  such that $\bar{\rho}(\Lambda_0 + K_0)
            \le \delta_0$, that is,
    the set of points in $\mathbb{R}^{n-kd}$ not covered by $K_0 + \Lambda_0$ has density at most $\delta_0$, where
    \begin{align*}
 %       \bar{\rho}(\Lambda_0 + K_0)        \le
         \delta_0 
        \coloneqq C_{\ref{LEMMA:almost-cover-Rogers}} (n-kd)^3 \left(16/27\right)^{\frac{n-kd}{4}}
        \le C_{\ref{LEMMA:almost-cover-Rogers}} n^3 \left(16/27\right)^{\frac{n}{5}}.
    \end{align*}
    By applying Lemma~\ref{LEMMA:cartesian-product-covering} iteratively $k$ times, we obtain lattices $\Lambda_1, \ldots, \Lambda_k$ such that, for each $i\in [k]$, the following properties hold: 
    \begin{itemize}
        \item the lattice $\Lambda_i \subseteq \mathbb{R}^{n-kd+id}$ satisfies $\det(\Lambda_{i}) = 1$, and 
        \item $\bar{\rho}(\Lambda_i + K_i) \le \delta_i \coloneqq C_{\ref{LEMMA:cartesian-product-covering}}\,  \delta_{i-1}^{2^d}$. 
        % the set of points in $\mathbb{R}^{n-dk+ik}$ not covered by $\Lambda_i+ K_{i}$ has density at most $\bar{\rho}(\Lambda_i + K_i) \le \delta_i \coloneqq C_{\ref{LEMMA:cartesian-product-covering}}  \delta_{i-1}^{2^d}$.
    \end{itemize}
    In particular, 
    \begin{align*}
        \delta_{k}
        = C_{\ref{LEMMA:cartesian-product-covering}}\, \delta_{k-1}^{2^d}
        = C_{\ref{LEMMA:cartesian-product-covering}}^{1+2^d} \delta_{k-2}^{2^{2d}}
        = \cdots 
        & = C_{\ref{LEMMA:cartesian-product-covering}}^{1+2^d+\cdots+2^{(k-1)d}} \delta_{0}^{2^{kd}} \\
        & = C_{\ref{LEMMA:cartesian-product-covering}}^{\frac{2^{kd}-1}{2^d-1}} \delta_{0}^{2^{kd}} 
        = C_{\ref{LEMMA:cartesian-product-covering}}^{\frac{-1}{2^d-1}} \Big(C_{\ref{LEMMA:cartesian-product-covering}}^{\frac{1}{2^d-1}} \delta_{0}\Big)^{2^{kd}}.
        % & = \beta^{\frac{-1}{2^d-1}} \left(\beta^{\frac{1}{2^d-1}} C n^3 \left(\frac{16}{27}\right)^{n/4} \right)^{2^{kd}}  
        % < \frac{1}{n^n+1}.
    \end{align*}
    Since $C_{\ref{LEMMA:almost-cover-Rogers}}$, $C_{\ref{LEMMA:cartesian-product-covering}}$, and $d$ are fixed, we can choose $n$ sufficiently large so that 
    \begin{align*}
        n 
        \ge \max\left\{ C_{\ref{LEMMA:cartesian-product-covering}}^{\frac{-1}{2^d-1}},~C_{\ref{LEMMA:almost-cover-Rogers}} C_{\ref{LEMMA:cartesian-product-covering}}^{\frac{1}{2^d-1}}\right\}.
    \end{align*}
    Combining it with the assumption $k d \ge \log_{2} \ln n + 4$, we obtain 
    \begin{align*}
        \ln \delta_{k} 
        \le \ln C_{\ref{LEMMA:cartesian-product-covering}}^{\frac{-1}{2^d-1}} + 2^{kd} \ln \left(C_{\ref{LEMMA:cartesian-product-covering}}^{\frac{1}{2^d-1}} \delta_{0}\right) 
        & \le \ln n + 2^{kd} \left( \ln n^4 + \ln \left(\frac{16}{27}\right)^{n/5} \right) \\
        & \le \ln n + 16 \ln n \left( 4\ln n - \frac{n}{5} \ln \frac{27}{16} \right) \\
        & =  - \left(\frac{16}{5} \ln \frac{27}{16}\right) n \ln n + 64 \ln^2 n + \ln n, 
    \end{align*}
    which is smaller than $- \ln \left(n^n + 1\right) = -(1+o(1)) n\ln n$ as $n$ is sufficiently large. 
    Thus,
    \begin{align*}
        \delta_{k}
        \le \left( n^n + 1 \right)^{-1}. 
    \end{align*}
    So it follows from Lemma~\ref{LEMMA:expand-Rogers} that $\Lambda_k+ \left(1+ 1/n\right) K_{k}=\mathbb{R}^n$. Since $T(K_{k}) \subseteq B$, we obtain 
    \begin{align*}
        \big(\Lambda_k, \left(1+1/n\right) T^{-1}(B) \big) = \mathbb{R}^{n}, 
    \end{align*}
    which implies that $\big(T(\Lambda_k), \left(1+1/n\right) B\big)$ forms a (sphere lattice) covering of $\mathbb{R}^{n}$. 

    It remains to show that the density of $\big(T(\Lambda_k), \left(1+1/n\right) B\big)$ gives the desired upper bound. 
    Indeed, by~\eqref{equ:k-choice},~\eqref{equ:eta-choice}, and~\eqref{equ:volume-B-upper-bound-b}, we have 
    \begin{align*}
        \Theta\big(T(\Lambda_k), \left(1+1/n\right) B\big)
        & = \frac{\vol\left((1+1/n)B\right)}{|\det (T(\Lambda_k))|}  
         =\left(1+\frac{1}{n}\right)^{n} \frac{\vol(B)}
        {|\det(T)||\det(\Lambda_k)|}  \\
        & \le 2 \mathrm{e} \eta  \left(\frac{D}{ \nu_{d} }\right)^k (2\pi \mathrm{e})^{\frac{kd}{2}} 
        \le C n \left(\frac{D}{ \nu_{d} }\right)^{\frac{1}{d}\log_2\ln n} (2\pi \mathrm{e})^{\frac{1}{2}\log_2\ln n}, 
    \end{align*}
    as claimed. 
    This completes the proof of Theorem~\ref{THM:general-parallelepiped}. 
\end{proof}

%%%%%%%%%%%%%%%%%%%%%%%%%%%%%%%%%%
\section{Proof of Lemma~\ref{LEMMA:cartesian-product-covering}}\label{SEC:proof-lemma-cartesian-product-covering}
In this section, we present the proof of Lemma~\ref{LEMMA:cartesian-product-covering}, starting with a few preliminary lemmas.

We use $\mathrm{dist}(\mathbf{x}, \mathbf{y})$ denote the Euclidean distance between two points $\mathbf{x}, \mathbf{y} \in \mathbb{R}^{d}$. 
Given a set $S \subseteq \mathbb{R}^{d}$ and a point $\mathbf{x} \in \mathbb{R}^{d}$, we define
\begin{align*}
    \mathrm{dist}(\mathbf{x}, S)
    \coloneqq \inf\left\{\mathrm{dist}(\mathbf{x}, \mathbf{y}) \colon \mathbf{y} \in S \right\}. 
\end{align*}

In the following lemma, we make no attempt to optimise $C_{\ref{LEMMA:point-type-upper-bound}}$ as a function of $d$.
%, as its value has no effect on our final density bound.

\begin{lemma}\label{LEMMA:point-type-upper-bound}
    Let $d \ge 1$ be an integer and $D \in [0, \nu_d]$ be a real number. For every robust lattice covering $(\Lambda, B^{d}_r)$ of $\mathbb{R}^{d}$ with density $D$, the number of fundamental parallelepipeds contained in $B_{2r}^d$ is at most $C_{\ref{LEMMA:point-type-upper-bound}} \coloneqq \left( 4^d d^{d/2}+1 \right)^{d \cdot 2^d}$.
\end{lemma}
\begin{proof}[Proof of Lemma~\ref{LEMMA:point-type-upper-bound}]
    By scaling $\Lambda$ if necessary, we may assume that $|\det(\Lambda)| = 1$. Consequently, $\vol(B^d_r) = D$. 
    Since $D \le \nu_{d}$, it follows that $r \le \sqrt{d}$. 
    Let 
    \begin{align*}
        C \coloneqq (4r)^d+1 \le 4^d d^{\frac{d}{2}}+1. 
    \end{align*}
    First, we show that the number of lattice points contained in $B_{2r}^{d}$ is bounded. 
    
    \begin{claim}\label{CLAIM:lattice-points-upper-bound}
        % If $B^d_{2r}(\mathbf{0})$ contains at least one fundamental parallelepiped, then the number of lattice points of $\Lambda$ contained in $B^d_{2r}(\mathbf{0})$ is at most $\left((4r)^d+1\right)^{d}$.
        We have 
        \begin{align*}
            \left|\Lambda \cap B^d_{2r}\right|
            \le C^{d}. 
        \end{align*}
    \end{claim}
    \begin{proof}[Proof of Claim~\ref{CLAIM:lattice-points-upper-bound}]
        % Since $P \coloneqq P'_{\mathbf{x}} \subseteq B^d_{2r}$, by the definition of fundamental parallelepiped, we can write $\Lambda$ as $\Lambda = \left\{ \sum_{i=1}^d \lambda_i  \mathbf{v}_i \colon \lambda_i \in\mathbb{Z} \right\}$. 
        By the definition of a robust lattice covering, the sphere $B_{r}^{d}$ contains a fundamental parallelepiped $\mathrm{P}_{\mathbf{x}}(\mathbf{v}_1, \ldots, \mathbf{v}_{d})$ of $\Lambda_{d}$, where $\mathbf{x} \in \mathbb{R}^{d}$ and $\mathbf{v}_1, \ldots, \mathbf{v}_{d} \in \mathbb{R}^{d}$ are linearly independent. 
        Note that $\mathbf{x}$ must lie in $\Lambda_{d}$, and hence the parallelepiped $\mathrm{P}_{\mathbf{0}}(\mathbf{v}_1, \ldots, \mathbf{v}_{d})$ obtained by translating $\mathrm{P}_{\mathbf{x}}(\mathbf{v}_1, \ldots, \mathbf{v}_{d})$ by $-\mathbf{x}$ is also  a fundamental parallelepiped of $\Lambda_{d}$.
        Moreover, since $\mathbf{x} \in B_{r}^{d}$, we have $\mathrm{P}_{\mathbf{0}}(\mathbf{v}_1, \ldots, \mathbf{v}_{d}) \subseteq B_{2r}^{d}$.
        
        Let $\mathrm{P} \coloneqq \mathrm{P}_{\mathbf{0}}(\mathbf{v}_1, \ldots, \mathbf{v}_{d})$. 
        Since $\mathrm{P}$ is a fundamental parallelepiped of $\Lambda_{d}$, the lattice $\Lambda$ is generated %\op{"can be" suggests that there is a choice that works  while here all is determined, so I would put "is"}
        by $\mathbf{v}_1, \ldots, \mathbf{v}_{d} \in \mathbb{R}^{d}$. 
        Let $\mathbf{v}_0 \coloneqq \mathbf{0}$ and $V_0 \coloneqq \left\{\mathbf{v}_{0}\right\}$. 
        For each $j\in [d]$, define
        \begin{align*}
            V_j 
            \coloneqq \left\{ \sum_{i=1}^{j} \lambda_i \mathbf{v}_i \colon \text{$\lambda_i \in \mathbb{R}$ for $i \in [j]$} \right\}
            \quad\text{and}\quad 
            d_j 
            \coloneqq \mathrm{dist}\left(\mathbf{v}_j, V_{j-1}\right). 
        \end{align*}
        By the definition of a fundamental parallelepiped, we have $\vol(\mathrm{P}) = |\det(\Lambda)| = 1$. It follows that 
        \begin{align}\label{equ:volume-P-product}
            \prod_{j = 1}^{d} d_{j}
            = \vol(\mathrm{P})
            = 1.
        \end{align}
        Since $\mathrm{P} \subseteq B^d_{2r}(\mathbf{0})$, we have $d_j \le \dist(\mathbf{v}_{j},\mathbf{v}_{j-1} )\le 4r$ for every $j \in [d]$. 
        Combining it with~\eqref{equ:volume-P-product}, we obtain 
        \begin{align}\label{equ:dj-lower-bound}
            d_i =\prod_{j\in [d]\atop j\not=i} \frac1{d_j}
            \ge \frac{1}{(4r)^{d-1}}
            \quad\text{for every $i \in [d]$.}
        \end{align}
        Given $j \in [0, d]$ and a $(d-j)$-tuple $(x_{j+1}, \ldots, x_d) \in \mathbb{Z}^{d-j}$, define 
        \begin{align}\label{equ:Lambda-d-restriction}
            \Lambda(x_{j+1}, \ldots, x_d)
            \coloneqq \left\{ \sum_{i=1}^{j} \lambda_i \mathbf{v}_i + \sum_{i=j+1}^{d} x_i \mathbf{v}_{i} \colon \text{$\lambda_i \in \mathbb{Z}$ for $i \in [j]$}\right\}.
        \end{align}
        We now prove by induction on $j$ that, for every $(x_{j+1}, \ldots, x_d) \in \mathbb{Z}^{d-j}$, the following holds:
        \begin{align}\label{equ:induction-lattice-points-upper-bound}
            \left| \Lambda(x_{j+1}, \ldots, x_d) \cap B^d_{2r}\right| 
            \le C^{j}. 
        \end{align}
        The base case $j = 0$ is clear, since for every $(x_{1}, \ldots, x_d) \in \mathbb{Z}^{d}$, the set $\Lambda(x_1, \ldots, x_d)$ contains only one point. 
        
        We now focus on the inductive step. Fix $j \ge 1$ and suppose for contradiction that there exists $(x_{j+1}, \ldots, x_d) \in \mathbb{Z}^{d-j}$ for which~\eqref{equ:induction-lattice-points-upper-bound} fails. 
        By the inductive hypothesis,~\eqref{equ:induction-lattice-points-upper-bound} holds when $j$ is replaced with $j-1$. 
        This implies that the number of choices of $x_{j} \in \mathbb{Z}$ for which 
        \begin{align*}
            \Lambda(x_j, x_{j+1}, \ldots, x_d) \cap B^d_{2r} \neq \emptyset 
        \end{align*}
        is more than $C$. 
        Therefore, there exist two integers $y_j, z_j \in \mathbb{Z}$ with $|y_j -z_j| > C$ such that 
        \begin{align*}
            \Lambda(y_j, x_{j+1}, \ldots, x_d) \cap B^d_{2r}
            \neq \emptyset 
            \quad\text{and}\quad 
            \Lambda(z_j, x_{j+1}, \ldots, x_d) \cap B^d_{2r}
            \neq \emptyset. 
        \end{align*}
        Fix $\mathbf{u}_1 \in \Lambda(y_j, x_{j+1}, \ldots, x_d) \cap B^d_{2r}$ and $\mathbf{u}_{2} \in \Lambda(z_j, x_{j+1}, \ldots, x_d) \cap B^d_{2r}$.
        It follows from definition~\eqref{equ:Lambda-d-restriction} that there exists integers $\lambda_{1}, \cdots, \lambda_{j-1}$ such that 
        \begin{align*}
            \mathbf{u}_{2} - \mathbf{u}_{1}
            = \sum_{i=1}^{j-1} \lambda_{i} \mathbf{v}_{i} + (z_j - y_j) \mathbf{v}_{j}
            = (z_j - y_j) \left(\mathbf{v}_j - \sum_{i=1}^{j-1} \frac{\lambda_{i}}{y_j - z_j} \mathbf{v}_{i}\right). 
        \end{align*}
        Note that $\sum_{i=1}^{j-1} \frac{\lambda_{i}}{y_j - z_j} \mathbf{v}_{i}$ lies in the subspace $V_{j-1}$. Therefore, by~\eqref{equ:dj-lower-bound}, we have
        \begin{align*}
            \dist(\mathbf{u}_1, \mathbf{u}_2)
            \ge |z_j - y_j| \cdot \dist\left(\mathbf{v}_j, V_{j-1}\right)
            > C \cdot d_j
            \ge \frac{(4r)^d+1}{(4r)^{d-1}}
            > 4r, 
        \end{align*}
        which contradicts the assumption that both $\mathbf{u}_1$ and $\mathbf{u}_2$ lie in $B_{2r}^{d}$. 
        This completes the proof of the inductive step, and thus establishes~\eqref{equ:induction-lattice-points-upper-bound}. 
        In particular, setting $j = d$ yields Claim~\ref{CLAIM:lattice-points-upper-bound}. 
    \end{proof}
    
    % Let $\alpha \coloneqq \left( (4r)^d+1 \right)^{d \cdot 2^d}$.
    % Recall that $D = \vol(B_r^d)$, so $r$ depends only on $d$ and $D$, and hence, $\alpha$ depends only on $d$ and $D$. 
    Note that each $d$-dimensional (fundamental) parallelepiped has $2^d$ vertices, so it follows from Claim~\ref{CLAIM:lattice-points-upper-bound} that the number of fundamental parallelepiped contained in $B_{2r}^d$ is at most $C^{d\cdot 2^d}$, which completes the proof of Lemma~\ref{LEMMA:point-type-upper-bound}.  
\end{proof}

Let $\mathbb{T}^n \coloneqq \mathbb{R}^n/\mathbb{Z}^n$ denote the $n$-dimensional torus. 
The following two lemmas routinely follow from standard results. For completeness, we include their proofs.
\begin{lemma}\label{LEMMA:random-translation}
    Let $K \subseteq \mathbb{R}^{n}$ be a convex body and let $\delta:=\bar{\rho}\left(\mathbb{Z}^{n} + K \right)$. 
    Let $\mathbf{y} \in \mathbb{T}^{n}$ be a point chosen uniformly at random, according to the Lebesgue measure restricted to the cube $[0,1)^n$. 
    Let $\tilde{K} \coloneqq K \cup (K+\mathbf{y})$. 
    Then 
    \begin{align*}
        \mathbb{E}\left[ \bar{\rho}\big( \mathbb{Z}^{n} + \tilde{K}  \big) \right]
        = \delta^{2}. 
    \end{align*}
\end{lemma}
\begin{proof}[Proof of Lemma~\ref{LEMMA:random-translation}]
    Let $\chi \colon \mathbb{R}^{n} \to \{0, 1\}$ be the characteristic function of $\mathbb{Z}^{n} + K$. 
    Then $\chi$ is periodic with period $1$ in each of the coordinates. 
    It follows from $\bar{\rho}\left(\mathbb{Z}^{n} + K \right) = \delta$ that 
    \begin{align}\label{equ:expectation}
        % \mathbb{E}_{\mathbf{x}}\left[ 1-\chi(\mathbf{x})\right] 
        \int_{[0,1)^{n}}  \left(1-\chi(\mathbf{x})\right) dx_1\cdots dx_n  
        = \delta. 
    \end{align}
    Suppose that $\mathbf{y} \in \mathbb{T}^{n}$ is a point chosen uniformly at random according to the Lebesgue measure restricted to the cube $[0,1)^n$, and $\tilde{K} = K \cup (K+\mathbf{y})$.
    Using~\eqref{equ:expectation}, we obtain 
    \begin{align*}
        \mathbb{E}\left[ \bar{\rho}\big( \mathbb{Z}^{n} + \tilde{K}  \big) \right]
        & = \int_{[0,1)^{n}} \left(\int_{[0,1)^{n}}  \left(1-\chi(\mathbf{x})\right) \left(1-\chi(\mathbf{x} + \mathbf{y})\right) dx_1\cdots dx_n \right) dy_1 \cdots dy_n \\
        & = \int_{[0,1)^{n}}  \left(1-\chi(\mathbf{x})\right) \left( \int_{[0,1)^{n}} \left(1-\chi(\mathbf{x} + \mathbf{y})\right) dy_1 \cdots dy_n \right) dx_1\cdots dx_n \\
        & = \int_{[0,1)^{n}}  \left(1-\chi(\mathbf{x})\right) \left( \int_{[0,1)^{n}} \left(1-\chi(\mathbf{z})\right) dz_1 \cdots dz_n \right) dx_1\cdots dx_n 
        = \delta^2,  
    \end{align*}
    as desired.     
    % If we pick a point $\mathbf{x} $ uniformly at random form  cube $[0,1)^n$ according to the Lebesgue measure, then we have 
    % \[
    % \mathbb{E}_{\mathbf{x}}\left[ 1-\chi(\mathbf{x})\right] = \int_0^1 \cdots \int_0^1  \left(1-\chi(\mathbf{x})\right) dx_1dx_2\dots dx_n  =\delta
    % \]
    % Then we derive that 
    % \begin{align*}
    % \mathbb{E}\left[ \bar{\rho}\left( \Lambda + \tilde{K}  \right) \right] &= \mathbb{E}_{\mathbf{y}}\mathbb{E}_{\mathbf{x}}\left[ \left(1-\chi(\mathbf{x})\right)\left((1-\chi(\mathbf{x+y})\right)\right] 
    %  = \mathbb{E}_{\mathbf{x}}\mathbb{E}_{\mathbf{y}}\left[\left(1-\chi(\mathbf{x})\right)\left((1-\chi(\mathbf{x+y})\right)\right]\\
    % &=\mathbb{E}_{\mathbf{x}} \left[\left(1-\chi(\mathbf{x})\right)\mathbb{E}_{\mathbf{y}}\left[(1-\chi(\mathbf{x+y})\right]\right]
    % =\mathbb{E}_{\mathbf{x}} \left[\left(1-\chi(\mathbf{x})\right)\mathbb{E}_{\mathbf{y}}\left[(1-\chi(\mathbf{y})\right]\right] \\
    % & = \mathbb{E}_{\mathbf{x}} \left[\left(1-\chi(\mathbf{x})\right)\delta\right] =\delta^2.
    % \end{align*}
\end{proof}

\begin{lemma}\label{LEMMA:measure-preserve}
    Let $n, d \ge 1$ be integers. 
    Suppose that $M \in \mathbb{Z}^{d \times d}$ is a matrix with $|\det(M)| = 1$. 
    Define the map $\psi \colon \mathbb{R}^{d \times n} \to \mathbb{R}^{d \times n}$ by $\psi(X) = MX$ for all $X \in \mathbb{R}^{d\times n}$. 
    Then the map $\phi$ induced by $\psi$ on $\mathbb{T}^{d\times n}$, i.e.  
    \begin{align*}
        \phi(X) := \psi(X) \Mod{\mathbb{Z}^{d \times n}}
        \quad\text{for every}\quad X \in \mathbb{T}^{d\times n}, 
    \end{align*}
    is bijective and (Lebesgue) measure-preserving. 
\end{lemma}
\begin{proof}[Proof of Lemma~\ref{LEMMA:measure-preserve}]
    Let $M_{\ast} = \mathrm{diag}(M, \cdots, M) \in \mathbb{Z}^{dn \times dn}$ be the matrix obtained by placing $n$ copies of the matrix $M$ along the diagonal.
    It is clear that $M_{\ast}$ is an integer matrix with $|\det(M_{\ast}) |= 1$. 
    By Cramer's rule, the inverse $M_{\ast}^{-1}$ of $M_{\ast}$ is also an integer matrix with $|\det(M^{-1}_{\ast})| = 1$. 
    
    Define the map $\psi_{\ast} \colon \mathbb{R}^{dn} \to \mathbb{R}^{dn}$ by $\psi(\mathbf{x}) = M\mathbf{x}$ for every $\mathbf{x} \in \mathbb{R}^{dn}$. %\op{Shall we use column vectors and write $\psi(\mathbf{x}) = M\mathbf{x}$? Then you will need to change $XM$ to $MX$, etc.}
    It is clear that $\psi$ and $\psi_{\ast}$ define the same linear map under the identification of $\mathbb{R}^{d \times n} $ with $\mathbb{R}^{dn}$.
    Let $\varphi$ be the map induced by $\psi_{\ast}$ on $\mathbb{T}^{d\times n}$, i.e. 
    \begin{align*}
        \varphi(\mathbf{x}) = \psi_{\ast}(\mathbf{x}) \Mod{\mathbb{Z}^{dn}}
        \quad\text{for every}\quad \mathbf{x} \in \mathbb{T}^{dn}. 
    \end{align*}
    Since $|\det(M_{\ast})| = 1$, it follows from standard results in Analysis (see e.g.~{\cite[Lemma~40.4]{Book06}}) that $\psi_{\ast}$ is measure-preserving. 
    Thus, if we can show that $\varphi$ is bijective, it will follow that $\varphi$ is also measure-preserving.

    We begin by proving that $\varphi$ is injective. 
    Suppose to the contrary that there exist two distinct points $\mathbf{x}, \mathbf{y} \in [0,1)^{dn}$ such that $\varphi(\mathbf{x}) = \varphi(\mathbf{y})$.  
    Then we have $\varphi(\mathbf{x}) - \varphi(\mathbf{y}) = \mathbf{0}$, which means that 
    \begin{align}\label{equ:psi-ast-kernel}
        \psi_{\ast}(\mathbf{x}-\mathbf{y}) 
        = \psi_{\ast}(\mathbf{x}) - \psi_{\ast}(\mathbf{y}) 
        \in \mathbb{Z}^{dn}. 
    \end{align}
    Since both $M_{\ast}$ and $M_{\ast}^{-1}$ are integer matrices, the map $\psi_{\ast}$ induces a bijection from $\mathbb{Z}^{dn}$ onto itself. 
    Combining it with~\eqref{equ:psi-ast-kernel}, we conclude that $\mathbf{x}-\mathbf{y} \in \mathbb{Z}^{dn}$, which contradicts the assumption that $\mathbf{x} \neq \mathbf{y}$ and $\mathbf{x}, \mathbf{y} \in [0,1)^{dn}$. 

    Next, we show that $\varphi$ is surjective.
    Take an arbitrary point $\mathbf{y} \in [0,1)^{dn}$. 
    Since $M_{\ast}$ is invertible, the inverse $\psi_{\ast}^{-1}(\mathbf{y})$ exists. 
    Let $\mathbf{x}$ be the unique point in $[0,1)^{dn}$ such that $\mathbf{x} - \psi_{\ast}^{-1}(\mathbf{y}) \in \mathbb{Z}^{dn}$. 
    Then we have 
    \begin{align*}
        \varphi(\mathbf{x})
        = \psi_{\ast}(\mathbf{x}) \Mod{\mathbb{Z}^{dn}} 
        & = \psi_{\ast}\left( \psi_{\ast}^{-1}(\mathbf{y})  + \mathbf{x} - \psi_{\ast}^{-1}(\mathbf{y}) \right) \Mod{\mathbb{Z}^{dn}}  \\
        & = \psi_{\ast}\left( \psi_{\ast}^{-1}(\mathbf{y}) \right) + \psi_{\ast}\left( \mathbf{x} - \psi_{\ast}^{-1}(\mathbf{y})\right) \Mod{\mathbb{Z}^{dn}} 
        = \mathbf{y}, 
    \end{align*}
    where the last equality holds because $\psi_{\ast}$ maps $\mathbb{Z}^{dn}$ into $\mathbb{Z}^{dn}$. 
    This proves that $\varphi$ is surjective, and hence completes the proof of Lemma~\ref{LEMMA:measure-preserve}. 
\end{proof}

We are now ready to prove Lemma~\ref{LEMMA:cartesian-product-covering}.
\begin{proof}[Proof of Lemma~\ref{LEMMA:cartesian-product-covering}]
Given a convex body $K \subseteq \mathbb{R}^{n}$ and a lattice $\Lambda \subseteq \mathbb{R}^{n}$, let $\delta:=\bar{\rho}(\Lambda + K)$. 
    % $\Lambda + K$ covers all of $\mathbb{R}^{n}$ except for a set of density at most $\delta$. 
    By applying a linear transformation to $\Lambda$ if necessary, we may assume that $\Lambda = \mathbb{Z}^n$. 
    Let $\{\mathbf{e}_i \colon i\in [n]\}$ be the standard basis of $\mathbb{R}^n$, that is, 
    \begin{align*}
        \mathbf{e}_1 = (1, 0, 0, \dots, 0)^{t}, \quad
        \mathbf{e}_2 = (0, 1, 0, \dots, 0)^{t}, \quad
        \ldots, \quad  
        \mathbf{e}_n = (0, 0, 0,\dots, 1)^{t}.
    \end{align*}
    Fix a robust lattice covering $(\Lambda_d, B_r^d)$ of $\mathbb{R}^{d}$ with density $D\le \nu_d$. 
    By scaling $\Lambda_d$ if necessary, we may assume that $|\det(\Lambda_d)| =1$ and hence, $r$ is such that $\vol(B^d_{r}) = D$. 
    Let $\mathbf{b}_1, \ldots, \mathbf{b}_d \in \mathbb{R}^{d}$ be linearly independent vectors such that 
    \begin{align*}
        \Lambda_d 
        = \left\{\sum^d_{i=1}\lambda_i \mathbf{b}_i \colon \text{$\lambda_i\in \mathbb{Z}$ for $i \in [d]$}\right\}. 
    \end{align*}
    For $i\in [n]$, let $\tilde{\mathbf{e}}_i \coloneqq (\mathbf{e}_i, \mathbf{0}) \in \mathbb{R}^{n+d}$ be the concatenation of $\mathbf{e}_i \in \mathbb{R}^{n}$ and $\mathbf{0} \in \mathbb{R}^{d}$. 
    % Let $\mathbb{T}^{n} \coloneqq \mathbb{R}^n/\mathbb{Z}^n$ denote the $n$-dimensional torus. 
    For each $j \in [d]$, choose a vector $\mathbf{y}_{i} \in \mathbb{T}^{n}$ uniformly at random according to the Lebesgue measure restricted to the cube $[0,1)^n$, and let $\tilde{\mathbf{b}}_j \coloneqq (\mathbf{y}_j, \mathbf{b}_j) \in \mathbb{R}^{n+d}$ be the concatenation of $\mathbf{y}_i$ and $\mathbf{b}_i$. 
    Define a new (random) lattice 
    \begin{align*}
        \tilde{\Lambda}(\mathbf{y}_1, \ldots, \mathbf{y}_{d}) 
        \coloneqq \Bigg\{ \sum_{i=1}^{n}\lambda_i \tilde{\mathbf{e}}_i +\sum_{j =1}^{d} \mu_j \tilde{\mathbf{b}}_j \colon \text{$\lambda_i\in \mathbb{Z}$ for $i \in [n]$ and $\mu_j \in\mathbb{Z}$ for $j \in [d]$} \Bigg\}.
    \end{align*}
    Note that
    \begin{align*}
        |\det(\tilde{\Lambda}(\mathbf{y}_1, \ldots, \mathbf{y}_{d}))| 
        = |\det (\Lambda)| \cdot |\det(\Lambda_d)| 
        = 1,
    \end{align*}
    which can be seen by expanding the determinant of the corresponding matrix along the first $n$ columns (with each having only one non-zero entry, namely the diagonal entry $1$).
    
    Let $\tilde{K} = K \times B_{r}^{d} \subseteq \mathbb{R}^{n+d}$. 
    Let $C_{\ref{LEMMA:point-type-upper-bound}}$ be the constant given in Lemma~\ref{LEMMA:point-type-upper-bound} and define $C_{\ref{LEMMA:cartesian-product-covering}} \coloneqq \left( \left(C_{\ref{LEMMA:point-type-upper-bound}} + 1\right) d \right)^{2^d - 1}$. 
    We will show that, with positive probability, the following event occurs: 
    \begin{align*}
        \bar{\rho}\big( \tilde{\Lambda}(\mathbf{y}_1, \ldots, \mathbf{y}_{d}) + \tilde{K} \big)
        \le C_{\ref{LEMMA:cartesian-product-covering}} \delta^{2^d},
    \end{align*}
    that is, the set of points $(\mathbf{y}_1, \ldots, \mathbf{y}_{d}) \in \mathbb{T}^{n\times d}$ for which this inequality holds has positive Lebesgue measure. For this we need some further definitions and  two auxiliary claims.

    Let $B \coloneqq \left[\mathbf{b}_{1}, \ldots, \mathbf{b}_{d}\right] \in \mathbb{R}^{d \times d}$.
    Let $\phi_{\mathbf{y}_1, \ldots, \mathbf{y}_{d}} \colon \Lambda_d \to \mathbb{R}^{n}$ be the linear map defined by 
    \begin{align}\label{equ:def-phi}
        \phi_{\mathbf{y}_{1}, \ldots, \mathbf{y}_{d}}(\mathbf{z})
        = \left[\mathbf{y}_{1}, \ldots, \mathbf{y}_{d}\right] B^{-1} \mathbf{z}
        \quad\text{for every}\quad \mathbf{z} \in \Lambda_{d}. 
    \end{align}
    In other words, $\phi_{\mathbf{y}_1, \ldots, \mathbf{y}_{d}}$ sends a lattice point $\mathbf{z} \in \Lambda_d$ to $\sum_{j=1}^{d} \mu_j \mathbf{y}_j \in \mathbb{R}^{d}$, where $(\mu_1, \ldots, \mu_d) \in \mathbb{Z}^{d}$ is the unique collection of integers such that $\mathbf{z} = \sum_{j=1}^{d} \mu_j \mathbf{b}_j$.
Define    \begin{align*}
        \mathcal{P}
        \coloneqq \left\{ \mathrm{P} \subseteq B_{2r}^{d} \colon \text{$\mathrm{P}$ is a fundamental parallelepiped with $\mathbf{0}$ as a vertex}\right\}. 
    \end{align*}
    Lemma~\ref{LEMMA:point-type-upper-bound} implies that $|\mathcal{P}| \le C_{\ref{LEMMA:point-type-upper-bound}}$. 
    
    For every fundamental $\Lambda_{d}$-parallelepiped $\mathrm{P}$, let $E_{\mathrm{P}}$ denote the event that 
    \begin{align*}
        \bar{\rho}\left(\Lambda+K^{\mathrm{P}}\right)
        \le C_{\ref{LEMMA:cartesian-product-covering}} \delta^{2^d}, 
        \quad\text{where}\quad 
        K^{\mathrm{P}}(\mathbf{y}_1, \ldots, \mathbf{y}_{d})
        \coloneqq \bigcup_{\mathbf{x} \in \mathrm{P}} \left( K+\phi_{\mathbf{y}_1, \ldots, \mathbf{y}_{d}}(\mathbf{x}) \right).
    \end{align*}
    Our goal is to show that for every $\mathrm{P} \in \mathcal{P}$, the event $E_{\mathrm{P}}$ occurs with high probability. We begin with the  case of $E_{\mathrm{P}_{\ast}}$, where, for convenience, we define 
    \begin{align*}
        \mathrm{P}_{\ast} \coloneqq \mathrm{P}_{\mathbf{0}}(\mathbf{b}_1, \ldots, \mathbf{b}_{d})
        \quad\text{and}\quad
        K_{\ast}(\mathbf{y}_1, \ldots, \mathbf{y}_{d}) \coloneqq \bigcup_{\mathbf{x} \in V(\mathrm{P}_{\ast})} \left( K + \phi_{\mathbf{y}_1, \ldots, \mathbf{y}_{d}}(\mathbf{x}) \right).
    \end{align*}
    \begin{claim}\label{CLAIM:prob-standard-parallelepiped}
        We have 
        \begin{align*}
            % \mathbb{P}\left[ \bar{\rho}\left(\Lambda + K_{\ast}\right) \le C_{\ref{LEMMA:cartesian-product-covering}} \delta^{2^d} \right] 
            \mathbb{P}\left[ E_{\mathrm{P}_{\ast}} \right]
            \ge 1- \frac{1}{C_{\ref{LEMMA:point-type-upper-bound}} + 1}. 
        \end{align*}
    \end{claim}
    \begin{proof}[Proof of Claim~\ref{CLAIM:prob-standard-parallelepiped}]
        For each collection $\mathbf{y}_1, \ldots, \mathbf{y}_{d} \in \mathbb{T}^{n}$, since $\phi_{\mathbf{y}_1, \ldots, \mathbf{y}_{d}} \colon \Lambda_d \to \mathbb{R}^{n}$ is a linear map, the set $K_{\ast}(\mathbf{y}_1, \ldots, \mathbf{y}_{d})$ is determined by the basis $\mathbf{b}_1, \ldots, \mathbf{b}_{d}$ of $\mathrm{P}_{\ast}$.
        Thus, define $K_{0}(\mathbf{y}_1, \ldots, \mathbf{y}_{d}) \coloneqq K$, and for each $i \in [d]$, let  
        \begin{align*}
            K_{i}(\mathbf{y}_1, \ldots, \mathbf{y}_{d})
            \coloneqq K_{i-1}(\mathbf{y}_1, \ldots, \mathbf{y}_{d}) \cup \big(K_{i-1}(\mathbf{y}_1, \ldots, \mathbf{y}_{d}) + \mathbf{y}_i \big). 
        \end{align*}
        It follows from the definition of $\phi_{\mathbf{y}_1, \ldots, \mathbf{y}_{d}}$ that $\phi_{\mathbf{y}_1, \ldots, \mathbf{y}_{d}}(\mathbf{b}_i) = \mathbf{y}_i$ for $i \in [d]$. 
        Therefore, $K_{\ast}(\mathbf{y}_1, \ldots, \mathbf{y}_{d}) = K_d(\mathbf{y}_1, \ldots, \mathbf{y}_{d})$. 
 
        Let $C \coloneqq C_{\ref{LEMMA:point-type-upper-bound}} + 1$. 
        For $i \in [0, d]$, define
        \begin{align*}
            \pmb{\rho}_i 
            \coloneqq \bar{\rho}\left( \Lambda + K_{i}(\mathbf{y}_1, \ldots, \mathbf{y}_{d}) \right)
            \quad\text{and}\quad 
            \delta_{i}
            \coloneqq \frac{\left( Cd\delta \right)^{2^{i}}}{Cd}, 
        \end{align*}
        and let $E_i$ denote the event that $\pmb{\rho}_i \le \delta_{i} = Cd \delta_{i-1}^{2}$. 
        By assumption, we have $\pmb{\rho}_0 = \delta = \delta_0$. 

        For each $i \in [d]$, since $\mathbf{y}_i$ is independent of $\mathbf{y}_1, \ldots, \mathbf{y}_{i-1}$, it follows from Lemma~\ref{LEMMA:random-translation} and Markov's inequality that 
        \begin{align*}
            \mathbb{P} \left[ \pmb{\rho}_i \ge \delta_i \mid \pmb{\rho}_{i-1} \le \delta_{i-1} \right] 
            \le \frac{\delta_{i-1}^{2}}{\delta_i}
            = \frac{\delta_{i-1}^2}{C d  \delta_{i-1}^2}
            = \frac{1}{C d}.
        \end{align*}
        Thus, the conditional probability satisfies
        \begin{align*}
            \mathbb{P}\left[E_i \mid E_{i-1} \right]
            \ge 1- \frac{1}{Cd}
            \quad\text{for every}\quad i \in [d]. 
        \end{align*}
        Combining this with Bernoulli's inequality, we obtain 
        \begin{align*}
            % \mathbb{P}\left[ \bar{\rho}\left(\Lambda + K_{\ast}\right) \le C_{\ref{LEMMA:cartesian-product-covering}} \delta^{2^d} \right] 
            \mathbb{E}\left[E_{\mathrm{P}_{\ast}}\right]
            = \mathbb{P}\left[ E_{d} \right]
            & \ge \mathbb{P}\left[ E_{0} \wedge \cdots \wedge E_{d} \right] \\
            &  = \prod_{i=1}^d  \mathbb{P}\left[ E_i \mid E_{0} \wedge \cdots \wedge E_{i-1} \right] \\
            & \ge \prod_{i=1}^d  \mathbb{P}\left[ E_i \mid E_{i-1} \right] 
            \ge \left( 1-\frac{1}{C d} \right)^d 
            \ge 1-\frac{1}{C}, 
        \end{align*}
        as desired. 
    \end{proof}%CLAIM

    Next, we extend the conclusion of Claim~\ref{CLAIM:prob-standard-parallelepiped} to all elements of $\mathcal{P}$.
    \begin{claim}\label{CLAIM:prob-each-type}
        For every $\mathrm{P} \in \mathcal{P}$, we have 
        \begin{align*}
            \mathbb{P}\left[ E_{\mathrm{P}} \right] \ge 1- \frac{1}{C_{\ref{LEMMA:point-type-upper-bound}} + 1}. 
        \end{align*}
        In particular, with positive probability, all of the events $\{E_{\mathrm{P}} \colon \mathrm{P} \in \mathcal{P}\}$ occur simultaneously.
    \end{claim}
    \begin{proof}
    %Since \[
    %A'_\mathbf{w} \coloneqq \left\{ \sum_{i=1}^d \lambda_i  \mathbf{w}_i, \lambda_i \in \{0,1\} \right\}
    %\]
            % \begin{align*}
        %     \left\{E_{\mathrm{P}} \colon \text{$\mathrm{P} \subseteq B_{2r}^{d}$ is fundamental parallelepiped containing $\mathbf{0}$ as a vertex}\right\}
        % \end{align*}
        Fix $\mathrm{P} \in \mathcal{P}$.
        Define sets 
        \begin{align*}
            S
            & \coloneqq \left\{ (\mathbf{y}_1, \cdots, \mathbf{y}_d) \in \mathbb{T}^{n} \times \cdots \times \mathbb{T}^{n}  \colon \text{$E_{\mathrm{P}_{\ast}}$ holds} \right\}
            \quad\text{and}\quad \\[0.4em] 
            T
            & \coloneqq \left\{ (\mathbf{y}_1, \cdots, \mathbf{y}_d) \in \mathbb{T}^{n} \times \cdots \times \mathbb{T}^{n}  \colon \text{$E_{\mathrm{P}}$ holds} \right\} 
        \end{align*}
        Note that $\mathbb{P}\left[ E_{\mathrm{P}_{\ast}} \right]$ and $\mathbb{P}\left[ E_{\mathrm{P}} \right]$ are equal to $\mu(S)$ and $\mu(T)$, respectively, where $\mu$ denotes the Lebesgue measure.  
        Recall from Claim~\ref{CLAIM:prob-standard-parallelepiped} that $\mathbb{P}\left[ E_{\mathrm{P}_{\ast}} \right] \ge 1 - \left(C_{\ref{LEMMA:point-type-upper-bound}} + 1\right)^{-1}$. 
        So it suffices to show that $\mu(T) \ge \mu(S)$ (In fact, a straightforward modification of the argument below shows that $S$ and $T$ have the same Lebesgue measure). 
        
        Fix linearly independent vectors $\mathbf{w}_1, \ldots, \mathbf{w}_{d} \in \mathbb{R}^{d}$ such that 
        \begin{align*}
            V(\mathrm{P})
            = \left\{ \sum_{i=1}^d \lambda_i  \mathbf{w}_i \colon \text{$\lambda_i \in \{0,1\}$ for $i \in [d]$} \right\}.
        \end{align*}
        For each collection $\mathbf{z}_1, \ldots, \mathbf{z}_{d} \in \mathbb{T}^{n}$, let  $K_{0, \mathrm{P}}(\mathbf{z}_1, \ldots, \mathbf{z}_{d}) \coloneqq K$, and for each $i \in [d]$, let 
        \begin{align*}
            K_{i, \mathrm{P}}(\mathbf{z}_1, \ldots, \mathbf{z}_{d})
            \coloneqq K_{i-1, \mathrm{P}}(\mathbf{z}_1, \ldots, \mathbf{z}_{d}) \cup \big( K_{i-1, \mathrm{P}}(\mathbf{z}_1, \ldots, \mathbf{z}_{d}) + \phi_{\mathbf{z}_1, \ldots, \mathbf{z}_{d}}(\mathbf{w}_i) \big). 
        \end{align*}
        Similar to the proof of Claim~\ref{CLAIM:prob-standard-parallelepiped}, we have $K_{\mathrm{P}}(\mathbf{z}_1, \ldots, \mathbf{z}_{d}) = K_{d, \mathrm{P}}(\mathbf{z}_1, \ldots, \mathbf{z}_{d})$. 

        Recall that $\{\mathbf{b}_1, \ldots, \mathbf{b}_{d}\}$ is a basis of $\Lambda_{d}$ and $B = \left[\mathbf{b}_1, \ldots, \mathbf{b}_{d}\right] \in \mathbb{R}^{d\times d}$. 
        Let $W \coloneqq \left[\mathbf{w}_1, \ldots, \mathbf{w}_{d}\right] \in \mathbb{R}^{d\times d}$, and let $M \in \mathbb{R}^{d \times d}$ be the matrix such that $WM = B$. 
        It follows that $MBW = I$, and thus,  
        \begin{align}\label{equ:M-transformation-Lambda-d}
            MB \mathbf{w}_i = \mathbf{e}_{i}. 
        \end{align}
        Given an element $(\mathbf{y}_1, \ldots, \mathbf{y}_{d}) \in S$, define the map $\varphi \colon \mathbb{T}^{d \times n} \to \mathbb{T}^{d \times n}$ by
        \begin{align*}
        \varphi([\mathbf{y}_1, \ldots, \mathbf{y}_{d}]^t):=
            \left[ \psi(\mathbf{y}_1), \ldots, \psi(\mathbf{y}_d) \right]^t,
        \end{align*}
        where
        \begin{align*}
            \left[ \psi(\mathbf{y}_1), \ldots, \psi(\mathbf{y}_d) \right]
            \coloneqq \left[ \mathbf{y}_1, \ldots, \mathbf{y}_{d} \right] M \Mod{\mathbb{Z}^{n\times d}}. 
        \end{align*}
        Combining~\eqref{equ:M-transformation-Lambda-d} with definition~\eqref{equ:def-phi}, we obtain   
        \begin{align*}
            \phi_{\psi(\mathbf{y}_1), \ldots, \psi(\mathbf{y}_d)}(\mathbf{w}_i)
            & = \left[\psi(\mathbf{y}_1), \ldots, \psi(\mathbf{y}_d)\right] B^{-1} \mathbf{w}_{i} \\
            & = \left[\mathbf{y}_{1}, \ldots, \mathbf{y}_{d}\right] M B^{-1} \mathbf{w}_{i}
            = \left[\mathbf{y}_{1}, \ldots, \mathbf{y}_{d}\right] \mathbf{e}_{i} 
            = \mathbf{y}_{i}. 
        \end{align*}
        It follows that 
        \begin{align*}
            K_{\mathrm{P}}\left(\psi(\mathbf{y}_1), \ldots, \psi(\mathbf{y}_d)\right)
            = K_{\ast}\left(\mathbf{y}_1, \ldots, \mathbf{y}_{d}\right). 
        \end{align*}
        Since $(\mathbf{y}_1, \ldots, \mathbf{y}_{d}) \in S$, it follows from the definition of $S$ that 
        \begin{align*}
            \bar{\rho}\big( \Lambda + K_{\mathrm{P}}\left(\psi(\mathbf{y}_1), \ldots, \psi(\mathbf{y}_d)\right) \big)
            = \bar{\rho}\big( \Lambda + K_{\ast}\left(\mathbf{y}_1, \ldots, \mathbf{y}_{d}\right) \big)
            \le C_{\ref{LEMMA:cartesian-product-covering}} \delta^{2^d},  
        \end{align*}
        which implies that $\left(\psi(\mathbf{y}_1), \ldots, \psi(\mathbf{y}_d)\right) \in T$. 

        Since $\mathrm{P}$ is a fundamental parallelepiped of $\Lambda_{d}$, the matrix $M$ must be an integer matrix with determinant $\pm1$. 
        %By reordering the basis of $\Lambda_{d}$, we can assume that $\det(M) = 1$.
        %For $d=1$ we need to negate the vector. 
        It follows from Lemma~\ref{LEMMA:measure-preserve} that the map $\varphi$ on $\mathbb{T}^{d\times n}$ is measure-preserving. 
        Therefore, 
        \begin{align*}
            \mathbb{P}\left[ E_{\mathrm{P}} \right] 
            = \mu(T) 
            \ge \mu(S)
            = \mathbb{P}\left[ E_{\mathrm{P}_{\ast}} \right] 
            \ge 1- \frac{1}{C_{\ref{LEMMA:point-type-upper-bound}} + 1}, 
        \end{align*}
        which proves Claim~\ref{CLAIM:prob-each-type}. 
    \end{proof}%CLAIM
    Fix a collection $(\mathbf{y}_1, \ldots, \mathbf{y}_{d} ) \subseteq \mathbb{T}^{n \times d}$ such that all of the events $\{E_{\mathrm{P}} \colon \mathrm{P} \in \mathcal{P}\}$ occur simultaneously. 
    For convenience, let $\tilde{\Lambda} \coloneqq \tilde{\Lambda}(\mathbf{y}_1, \ldots, \mathbf{y}_{d})$ and $\phi \coloneqq \phi_{\mathbf{y}_1, \ldots, \mathbf{y}_{d}}$.
    
    \hide{For every point $\mathbf{z} \in \mathbb{R}^d$, define \op{I am confused by this definition: what is the role of $z$ in $V_{\mathbf{z}}$? Can we just use $\mathbb{R}^{n}$ everywhere instead of $V_{\mathbf{z}}$?}  
    \begin{align*}
        % V_{\mathbf{z}} 
        % \coloneqq \left\{ (\mathbf{x}, \mathbf{z}) \colon (\mathbf{x}, \mathbf{z}) \in \mathbb{R}^{n+d} \right\}
        % \quad\text{and}\quad 
        V_{\mathbf{z}}
        \coloneqq \left\{ \mathbf{x} \colon (\mathbf{x}, \mathbf{z}) \in \mathbb{R}^{n+d} \right\}
        \cong \mathbb{R}^{n}. 
    \end{align*}
    }
    For every lattice point $\mathbf{z} \in \Lambda_d$, define 
    \begin{align*}%\label{equ:Def-Lambda-z}
        % \Lambda_{\mathbf{z}}  
        % \coloneqq \left\{ (\mathbf{x}, \mathbf{z}) \colon (\mathbf{x}, \mathbf{z}) \in\tilde{\Lambda} \right\}
        % \quad\text{and}\quad 
        \Lambda_{\mathbf{z}}
        \coloneqq \left\{ \mathbf{x} \colon (\mathbf{x}, \mathbf{z}) \in\tilde{\Lambda} \right\}. 
    \end{align*}
    Note that, by definition, for every $\mathbf{z} \in \Lambda_d$, the lattice $\Lambda_{\mathbf{z}}$ can be written as 
    \begin{align}\label{equ:Lambda-z}
        \Lambda_{\mathbf{z}}
        = \Lambda + \phi(\mathbf{z}).  
    \end{align}
    This means that $\Lambda_{\mathbf{z}}$ is simply a translation of the lattice $\Lambda$. 
    % Therefore, $(\Lambda_{\mathbf{z}}, K)$ is a lattice covering of $V_{\mathbf{z}}$. 
    
    Fix a point $\mathbf{w} \in \mathbb{R}^{d}$ and fix a fundamental parallelepiped $\mathrm{P}_{\mathbf{w}}$ that is contained in the ball $B^{d}_{r}(\mathbf{w})$.
    Fix a point $\mathbf{s} \in V(\mathrm{P}_{\mathbf{w}})$ and let $\mathrm{P}_{\mathbf{w}}' \coloneqq \mathrm{P}_{\mathbf{w}} - \mathbf{s}$ be obtained from $\mathrm{P}_{\mathbf{w}}$ by translating by $-\mathbf{s}$. 
    Note that $\mathrm{P}_{\mathbf{w}}'$ is a fundamental parallelepiped contained in $B_{2r}^{d}$ and $\mathbf{0}$ is a vertex in $\mathrm{P}_{\mathbf{w}}'$, that is, $\mathrm{P}_{\mathbf{w}}' \in \mathcal{P}$. 
    
    Consider the following restriction of the set $\tilde{K}+ \tilde{\Lambda}$:
    \begin{align*}
        (\tilde{K}+ \tilde{\Lambda})\mid_{\mathbf{w}}
        \coloneqq \left\{\mathbf{x} \in \mathbb{R}^{n} \colon (\mathbf{x}, \mathbf{w}) \in \tilde{K}+ \tilde{\Lambda} \right\}.
    \end{align*}
    Since $\tilde{K} = K \times B_{r}^{d}$, a point $\mathbf{x} \in \mathbb{R}^{n}$ is covered by $(\tilde{K}+ \tilde{\Lambda})\mid_{\mathbf{w}}$ if there exists a point $\tilde{\mathbf{w}} \in \Lambda_{d} \cap B_{r}(\mathbf{w})$ such that $\mathbf{x} \in \Lambda_{\tilde{\mathbf{w}}} + K$. 
    In particular, the set $\bigcup_{\tilde{\mathbf{w}} \in V(\mathrm{P}_{\mathbf{w}})} \left( \Lambda_{\tilde{\mathbf{w}}} + K \right) \subseteq \mathbb{R}^{n}$ is covered by $(\tilde{K}+ \tilde{\Lambda})\mid_{\mathbf{w}}$.
    By~\eqref{equ:Lambda-z} and the linearity of $\phi$, we have %\op{Do we need to replace $P$ by $V(P)$ in some places below?}
    \begin{align*}
        \bigcup_{\tilde{\mathbf{w}} \in V(\mathrm{P}_{\mathbf{w}})} \left( \Lambda_{\tilde{\mathbf{w}}} + K \right)
        & = \bigcup_{\tilde{\mathbf{w}} \in V(\mathrm{P}_{\mathbf{w}})} \left(\Lambda+ \phi(\tilde{\mathbf{w}}) +K \right) \\
        & = \bigcup_{\tilde{\mathbf{w}} \in V(\mathrm{P}_{\mathbf{w}})} \left(\Lambda+ \phi(\tilde{\mathbf{w}} - \mathbf{s} + \mathbf{s}) +K \right) \\
        & = \phi(\mathbf{s}) + \Lambda +  \bigcup_{\tilde{\mathbf{w}} \in V(\mathrm{P}'_\mathbf{w})} \left( K+\phi(\tilde{\mathbf{w}}) \right)
        = \phi(\mathbf{s}) + \Lambda + K^{\mathrm{P}'_\mathbf{w}}.
    \end{align*}
    Since $\mathrm{P}_{\mathbf{w}}' \in \mathcal{P}$, the choice of vectors $\{\mathbf{y}_1, \ldots, \mathbf{y}_{d}\}$ guarantees that 
    \begin{align*}
        \bar{\rho}\big( \Lambda + K^{\mathrm{P}'_\mathbf{w}} \big)
        \le C_{\ref{LEMMA:cartesian-product-covering}} \delta^{2^d}. 
    \end{align*}
    Since translation (by $\phi(\mathbf{s})$) does not affect the density of uncovered points, we also have
    \begin{align*}
        \bar{\rho}\big( \textstyle\bigcup_{\tilde{\mathbf{w}} \in V(\mathrm{P}_{\mathbf{w}})} \left( \Lambda_{\tilde{\mathbf{w}}} + K \right) \big)
        \le C_{\ref{LEMMA:cartesian-product-covering}} \delta^{2^d}. 
    \end{align*}
    Since $\mathbf{w} \in \mathbb{R}^{d}$ was arbitrary, it follows that $\bar{\rho} \big(\tilde{\Lambda}+ \tilde{K}\big) \le C_{\ref{LEMMA:cartesian-product-covering}} \delta^{2^d}$. 
    This completes the proof of Lemma~\ref{LEMMA:cartesian-product-covering}. 
\end{proof}

%%%%%%%%%%%%%%%%%%%%%%%%%%
\section{Concluding remarks}\label{SEC:remarks}

Rogers' original proof in~\cite{59R} is essentially the same as our proof of Theorem~\ref{THM:general-parallelepiped} in the special case $d=1$. We hope that the idea of using an iterative step where the dimension increases by more than $1$ will lead to further improvements (via Theorem~\ref{THM:general-parallelepiped} or some other estimates).

Unfortunately, we have little intuition about the optimal robust lattice covering density in dimensions $d\ge 4$. The presented constant $\beta$ in Theorem~\ref{THM:main-improve-Rogers} was the best one that came from our sporadic search for $d\le 3$. So the natural question motivated by Theorem~\ref{THM:general-parallelepiped} is the following. 
\begin{problem}\label{PROB:robust-covering-ratio}
    Determine $\tilde{\Theta}_{n}$. 
    In particular, what is the infimum of  
    \begin{align*}
        \frac{1}{n} \log_{2}\big(\tilde{\Theta}_{n}/\nu_{n}\big)?
    \end{align*}
\end{problem}

A lower bound $\tilde{\Theta}_{n} \ge {\nu_{n}}/{2^n}$ can be established via the following argument. 
Let $(\Lambda, B_{r}^{n})$ be a robust lattice covering of $\mathbb{R}^{n}$ with $\det(\Lambda) =1$. 
By definition, $B_{r}^{n}$ contains a fundamental parallelepiped $\mathrm{P}$.
It is not hard to show that the largest volume of a parallelepiped (not necessarily a $\Lambda$-parallelepiped) contained in $B_{r}^{n}$ is $\left(2r/\sqrt{n}\right)^{n}$, attained by an inscribed cube centred at the origin.
Since $\mathrm{P} \subseteq B_{r}^{n}$ is a parallelepiped with volume $\det(\Lambda) = 1$,  it follows that 
\begin{align*}
    \left(2r/\sqrt{n}\right)^{n} 
    \ge \vol(\mathrm{P})
    = |\det \Lambda|
    = 1, 
\end{align*}
which implies that $r \ge \sqrt{n}/2$. 
Therefore, we obtain the bound
\begin{align*}
    \tilde{\Theta}_{n}
    \ge \frac{\vol(B_{r}^{n})}{|\det(\Lambda)|}
    \ge \vol\left(B_{\sqrt{n}/2}^{n}\right)
    \ge \frac{1}{2^n} \vol\left(B_{\sqrt{n}}^{n}\right)
    = \frac{\nu_{n}}{2^n}.
\end{align*}

In particular, this yields $\frac{1}{n} \log_{2}\left(\tilde{\Theta}_{n}/\nu_{n}\right) \ge -1$, which means that the best possible value we can hope for $\gamma$ via an application of Theorem~\ref{THM:general-parallelepiped} as stated is at least $\frac{1}{2}\log_{2}(2\pi \mathrm{e}) - 1 = 1.0471...$\,.

%%%%%%%%%%%%%%%%%%%%%%%%%%%%%%%%%%%%%%%%%%%%%%
\bibliographystyle{abbrv}%abbrv
\bibliography{SphereC}
%%%%%%%%%%%%%%%%%%%%%%%%%%%%%%%%%%%%%%%%%%%%%%
\end{document}